\documentclass[leqno,twoside,11pt]{amsart}
\usepackage{latexsym,esint,url}

\def\endproof{\hspace*{\fill}\mbox{\ \rule{.1in}{.1in}}\medskip }
\def\t{\theta}
\def\f{\phi}
\def\D{\Delta}
\def\na{\nabla}

\def\r{\rho}
\def\ra{\vec{\rho}}
\def\bb{\bar B}
\def\bn{\bar N}
\def\bu{\bar U}
\def\n2{\|^2_{L^2}}

\newcommand\br{\begin{rem}}
\newcommand\er{\end{rem}}
\newcommand\bp{\begin{pmatrix}}
\newcommand\ep{\end{pmatrix}}
\newcommand\be{\begin{equation}}
\newcommand\ee{\end{equation}}
\newcommand\ba{\begin{equation}\begin{aligned}}
\newcommand\ea{\end{aligned}\end{equation}}

\newcommand{\CalP}{\mathcal{P}}

\newtheorem{theo}{Theorem}[section]

\newtheorem{lem}[theo]{Lemma}

\newtheorem{rem}[theo]{Remark}

\numberwithin{equation}{section}

\begin{document}

\title[Stokes-Boussinesq equation in a channel]{A stability result for  the 
Stokes-Boussinesq equations in infinite 3d channels}
\author{Marta Lewicka and Mohammadreza Raoofi}
\address{Marta Lewicka,  University of Pittsburgh, Department of Mathematics, 
301 Thackeray Hall, Pittsburgh, PA 15260, USA}
\address{Mohammadreza Raoofi, Isfahan University of Technology,
  Isfahan,Iran and
School of Mathematics, Institute for Research in Fundamental
Sciences (IPM),  P.O. Box 19395-5746, Tehran, Iran}
\email{lewicka@math.pitt.edu, raoofi@cc.iut.ac.ir}

\begin{abstract}
We consider the Stokes-Boussinesq (and the stationary Na\-vier-Stokes-Boussinesq) 
equations in a slanted, i.e. not aligned
with the gravity's direction, 3d channel and with an arbitrary
Rayleigh number. For the front-like initial data and under the 
no-slip boundary condition for the flow and
no-flux boundary condition for the reactant temperature, we
derive  uniform estimates on the burning rate and the flow velocity, which can be
interpreted as stability results for the laminar front.
 \end{abstract}

\maketitle

\section{Introduction and the main results}\label{introduction}

The Boussinesq system for a reactive flow is a  model describing 
flame propagation in a gravitationally stratified medium
\cite{ZBLM85}. It consists of the reaction-advection-diffusion
equation for the reactant temperature $\theta$ (normalised  so that
$0\leq\theta\leq 1$) and the Navier-Stokes
equations for the evolution of the incompressible flow $u$. The
variables $u$ and $\theta$ are coupled through the advection velocity
in the reaction equation, and
through the force term in the fluid equation.
After passing to nondimensional variables
\cite{BCR06}, the simplified Stokes-Boussinesq system takes the form:
\begin{align}
&\t_t +u\cdot\nabla\t-\D\t=f(\t), \label{heat}\\
&u_t -\nu\D u +\na p=\t\ra,  \label{fluid}\\
&\text{div}\, u =0. \label{solenoid}
\end{align}
Here, $\nu>0$ is the Prandtl number (inversely proportional to the
Reynolds number), while the reaction rate is given by a nonnegative,
nonlinear function $f(\t)$  of ignition type. That is, $f$ is
Lipschitz continuous, and there is a threshold temperature $0<\vartheta_0<1$  such that:
 \begin{equation}\label{f}
 f(\t)=0 \mbox{ for } \t\le \vartheta_0 \mbox{ and } \t\ge 1, \qquad
 f(\t)>0 \mbox{ on } (\vartheta_0, 1), \qquad f'(1)<0.
 \end{equation}
The vector $\ra=\rho \vec g$ corresponds
to the non-dimensional gravity $\vec g$ scaled by the Rayleigh number
$\rho$, and we assume that $\vec\rho$ is non-parallel to the unbounded
direction of the 3d channel $D$, which amounts to studying the system
(\ref{heat}) - (\ref{solenoid}) in:
$$D=(-\infty, \infty)\times \Omega = \{(x,\tilde x); ~ x\in
\mathbb{R}, ~ \tilde x\in\Omega\},$$
with:
$$\vec\rho \cdot e_3 \neq 0.$$
The crossection $\Omega$ is a sufficiently regular,
 connected and bounded domain in $\mathbb{R}^2$.
For the solutions to (\ref{heat}) - (\ref{solenoid}), we impose the Neumann  condition in
$\t$, and the no-slip (Dirichlet) boundary condition in $u$:
\begin{equation}\label{bcond}
\frac{\partial \theta}{\partial\vec n} = 0 \quad \mbox{ and } \quad
u=0 \mbox{ on } \partial D.
\end{equation}

The classical result by Kanel \cite{K61, B03} states that there exists
the unique speed $c_0$ (necessarily positive) and a (necessarily
decreasing) traveling wave profile $\Phi:\mathbb{R}\longrightarrow
[0,1]$, satisfying:
\begin{equation*}\label{laminar}
 -c_0\Phi' - \Phi'' = f(\Phi), \quad \Phi(-\infty)=1, \quad \Phi(+\infty)=0.
\end{equation*}
We call $c_0$ and $\Phi$ the {\em laminar speed} and {\em laminar
  front}, and to fix
the ideas, we let $\Phi(0) = \vartheta_0$. Then $\theta(x,t) =
\Phi(x-c_0t)$ is the unique (up to translations) traveling-wave
solution to:
\begin{equation}\label{oned}
\theta_t - \theta_{xx} = f(\theta), \qquad \theta (x,0) = \theta_0(x)
\end{equation}
joining the two equilibria $1$ and $0$.
Moreover, if the initial data $\theta_0(x)$ differs from the Heaviside
function $H(x)$ by a compactly supported error, then there exists a
shift $x_0$ such that, for the solution $\theta$ to (\ref{oned}) there holds:
$$\|\theta(\cdot, t) - \Phi(\cdot + c_0t + x_0)\|_{L^\infty(\mathbb{R})}
\to 0  \quad \mbox{ as } t\to\infty.$$
Our purpose is to reproduce this result for the problem (\ref{heat}) -
(\ref{solenoid}), (\ref{bcond}).
Following \cite{CKR03} and \cite{BCR06}, we define the \emph{bulk
burning rate} $\bar B(t)$ and the \emph{Nusselt number} 
$\bar N(t)$  by:
\begin{equation} \label{bn}
\begin{aligned}
&B(t) = \frac{1}{|\Omega|}\int_D f(\t)~\mbox{d}x~\mbox{d}\tilde x, 
\qquad &\bb(t) = \frac1t \int_0^tB(s)~\mbox{d}s\\
&N(t) = \frac{1}{|\Omega|}\int_D |\na\t|^2~\mbox{d}x~\mbox{d}\tilde x, 
\qquad &\bn(t) = \frac1t \int_0^tN(s)~\mbox{d}s,
\end{aligned}
\end{equation}
and the \emph{average  flow} $\bar U(t)$ by
\begin{equation} \label{ah}
 \bu(t) = \frac1t \int_0^t\|u(\cdot,s)\|_\infty ~\mbox{d}s.
\end{equation}

\begin{theo}\label{main}
Assume that the initial temperature $\t_0(x, \tilde x) \in [0,1]$ is
such that  $\theta_0(x,\tilde x) - H(x)$ is compactly supported in $D$.
Let $u_0\in W^{1,2}(D)$ and let $(\t, u, p)$ be a global solution of
\eqref{heat} - \eqref{solenoid}, \eqref{bcond} with:
$\theta(\cdot, 0) = \theta_0$ and $u(\cdot, 0) = u_0$.
Then, there exists a constant $C_\Omega$, depending only on $\Omega$, 
such that as $t\to+\infty$:
\begin{align}
&c_0 - C_\Omega\left(\frac{\r}{\nu}+\frac{\r^2}{\nu^2}\right) - o(1)\le
\bb(t) \le c_0 +
C_\Omega\left(\frac{\r}{\nu}+\frac{\r^2}{\nu^2}\right) + o(1), \label{estb}\\
&\bn(t)\le \left(C_\Omega\frac{\r}{\nu}
  +\sqrt{\frac{c_0}{2}+C_\Omega\frac{\r^2}{\nu^2}}\,\right)^2 
+ o(1),\label{estn}\\
&\bu(t)\le C_\Omega\left(\frac{\r}{\nu}+\frac{\r^2}{\nu^2}\right) + o(1).\label{estu}
\end{align}
\end{theo}
The above result shows that the solution of the
initial-boundary problem for \eqref{heat} - \eqref{solenoid}, with
small $\rho/\nu$ propagates
with finite speed close to the laminar front speed. Also, note that if
we replace $\t$ with the laminar front $\Phi$, then a simple
integration of \eqref{oned}  gives: 
$ \bar B(t)=c_0.$
This corresponds with the estimates in Theorem \ref{main} 
when $\rho = 0$, i.e. the system
 \eqref{heat} - \eqref{solenoid} turns out to be 
a regular perturbation of the reaction--diffusion equation.

Our next result states that with front-like initial data,
the solution to the studied system stays front-like:
\begin{theo}\label{secondtheo}
With the same assumptions as in Theorem \ref{main}, we have:
\begin{equation}
\begin{split}
&\Phi\left(x-c_0t+x_0+\bu(t)t +C_{\Omega,0}\sqrt{t}\right)-\frac{C_{\Omega,0}}{\sqrt{t}}\le \t(x,\tilde
x, t) \\ & \qquad\qquad\qquad
\le \Phi\left(x-c_0t-x_0+\bu(t)t +C_{\Omega,
    0}\sqrt{t}\right)+\frac{C_{\Omega,0}}{\sqrt{t}} 
\qquad \forall t>> 1,
\end{split}
\end{equation}
with appropriate $x_0>0$ and $C_{\Omega,0}$ depending on $\Omega$, $f$
and the initial data $\theta_0$.
\end{theo}

Finally, we have:

 \begin{theo}\label{remi}
The results in Theorem \ref{main} and Theorem \ref{secondtheo} remain
valid in either of the following cases:
\begin{enumerate}
\item[(i)] the channel $D=(-\infty, +\infty)\times [0,\lambda]$ is
  2-dimensional
\item[(ii)] the flow equation (\ref{fluid}) is replaced by the stationary
  Navier-Stokes:
$$-\nu\Delta u + u\cdot \nabla u + \nabla p = \theta\vec\rho,$$
and the crossection $\Omega$ is sufficiently ``thin'', i.e it
satisfies the same type of condition as in \cite{LM09}:
$$\frac{\sqrt{3}}{2}\frac{\rho}{\nu\sqrt{\pi \nu}} |\Omega|^{1/2} C_P \left(C_{PW} +
  \left(\fint|\vec g\cdot (0,\tilde x)|^2~\mbox{d}\tilde x\right)^{1/2}
\right) < 1.$$
\end{enumerate}
\end{theo}

Existence of traveling fronts for the Navier-Stokes-Boussinesq system
 in infinite channels has been the subject  of research during recent years
 \cite{CKR03, BCR06, TV01, CLR06, Lew07, LM09}.  Of interest has been also an
 understanding of the regularizing and mixing effect of convection
\cite{CNR08, CKRZ08}.  In \cite{CKR03} and \cite{BCR06}, the
 solutions of the system with front-like datum in a 2d strip have
been considered and uniform estimates for the full
Navier-Stokes-Boussinesq system
have been obtained for the \emph{stress-free}
boundary conditions on $u$.
In this paper we generalize these results to dimension 3 and the
more physically relevant \emph{no-slip} boundary conditions.
The analysis follows \cite{BCR06} closely; it first
seeks bounds for $\bn$ and $\bb$ using the
parabolic equation \eqref{heat}. Where our argument diverges 
from that of  \cite{BCR06} is in finding an estimate for $\|u\|_{\infty}$. 
In \cite{BCR06}, this has been done using
Poincar\'e's inequality for  vorticity, 
based on the assumption that the vorticity vanishes at
the boundary, and using the rather simple form of the vorticity
equation in dimension two. 
In the present case, we rely on the almost-uniform Xie bound 
in Lemma \ref{supest}.

\bigskip

\noindent{\bf Acknowledgments.} 
M.L. was partially supported by the NSF grants DMS-0707275 and DMS-0846996,
and by the Polish MN grant N N201 547438.
M.R. was partially supported by the IPM grant no 90350025.


\section{Auxiliary results}\label{estimates}

Our first result is a technical lemma, which extends Lemma
4.3 in \cite{BCR06}, proved there for 2d channels
(see also \cite{CKRZ08}). Here we prove a similar 
result in dimension 3, showing that diffusion in a tube
behaves like 1d heat equation.
\begin{lem}\label{estymata}
Let $u\in W^{1,2}(D,\mathbb{R}^3)$ with $\mathrm{div } ~u = 0$ be a given
solenoidal flow satisfying $u=0$ on $\partial D$. 
Let $\phi$ be the solution to the advection-diffusion equation:
\begin{equation} \label{phi}
\begin{aligned}
&\f_t+u\cdot\na\f-\D\f = 0   \quad  \mbox{ for } (x,\tilde x) \in D \mbox{
  and } t>0\\
&\frac{\partial\phi}{\partial \vec n} (x,\tilde x, t)= 0 \quad \mbox{
  for } (x,\tilde x) \in \partial D\\
&\f(x,\tilde x, 0)=\f_0(x,\tilde x).
\end{aligned}
\end{equation}
Then there exists a constant $C_\Omega$ depending only on $\Omega$ (in
particular independent of $u$ and $\f_0$) such that, for all
sufficiently large $t$, there holds:
\begin{equation*}\label{fee}
 \|\f(\cdot,t)\|_\infty\le\frac{C_\Omega}{\sqrt{t}} \|\f_0\|_{L^1(D)} \quad \forall
 t >> 1.
\end{equation*}
\end{lem}
\begin{proof}
{\bf 1.} We first prove a Nash-type inequality, valid for solutions of (\ref{phi}):
\begin{equation}\label{nash}
\|\na \phi(\cdot, t)\|_{L^2(D)}^2 \ge
C_\Omega\frac{\|\phi(\cdot, t)\|_{L^2}^6}{\|\phi(\cdot, t)\|^4_{L^1} +
  \|\phi(\cdot, t)\|_{L^1}\|\phi(\cdot, t)\|^3_{L^2}}  \qquad \forall t.
\end{equation}
To simplify the notation, in what follows we suppress the dependence
on $t$ and write $\phi$ instead of $\phi(\cdot, t)$, etc.

Define  $\psi(x)=\frac{1}{|\Omega|}\int_\Omega \phi(x, \tilde x) ~\mbox{d}\tilde x$ and
let $k(x,\tilde x)= \phi (x, \tilde x)-\psi (x)$, so that $\phi=\psi+k$
and $\int_\Omega k(x, \tilde x) ~\mbox{d}\tilde x=0$
for every $x\in\mathbb{R}$. Notice that, by Cauchy-Schwarz inequality:
\begin{equation}\label{2.6}
\|\psi\|^2_{L^2(D)}= |\Omega| \int_{\mathbb{R}} \left(\fint_\Omega \phi
~\mbox{d}\tilde x\right)^2 ~\mbox{d}x \leq \int_{\mathbb{R}} \int_\Omega
|\phi|^2 = \|\phi\|^2_{L^2(D)},
\end{equation}
Similarly:
\begin{equation}\label{2.6a}
\begin{split}
& \|\psi'\|_{L^2(D)}\le \|\phi_x\|_{L^2(D)}, \quad
\|\psi\|_{L^1(D)}\le \|\phi\|_{L^1(D)},\\
& \|k\|_{L^2}\le 2\|\phi\|_{L^2}, \quad \|\na k\|_{L^2}\le
2\|\na\phi\|_{L^2}, \quad  \|k\|_{L^1}\le 2\|\phi\|_{L^1}.
\end{split}
\end{equation}
Let $\hat\psi:\mathbb{R}\longrightarrow \mathbb{C}$ be the Fourier
transform of $\psi$, i.e.
$\hat\psi (\omega) = \int_\mathbb{R}\psi(s) e^{-2\pi i\omega s}
~\mbox{d}s.$
By Plancherel's identity, we have:
\begin{equation}\label{2.8}
\|\psi\|^2_{L^2(\mathbb{R})}= \|\hat\psi\|^2_{L^2(\mathbb{R})}, \qquad
\|\psi'\|^2_{L^2(\mathbb{R})}= \|2\pi i\omega\hat\psi\|^2_{L^2(\mathbb{R})}. 
\end{equation}
For a given positive $m$, we now write:
\begin{equation*}
\begin{split}
\|\psi\|^2_{L^2(\mathbb{R})} &=
\int_{|\omega|\le m}|\hat\psi(\omega)|^2~\mbox{d}\omega 
+ \int_{|\omega|> m}|\hat\psi(\omega)|^2~\mbox{d}\omega\\
&\leq 2 m \|\psi\|^2_{L^1(\mathbb{R})} + \frac{1}{m^2}
\int_\mathbb{R}\omega^2|\hat\psi(\omega)|^2 ~\mbox{d}\omega \\ &\leq
2 m \|\psi\|^2_{L^1(\mathbb{R})} + 
 \frac{1}{4\pi^2m^2}\|\psi'\|^2_{L^2(\mathbb{R})}
\end{split}
\end{equation*}
where the estimate of the first term follows by: $|\hat\psi(\omega)
|\leq \|\psi\|_{L^1(D)}$, while to estimate the second term we used  (\ref{2.8}).

Setting $m= ||\psi'||^{2/3}_{L^2(\mathbb{R})}||\psi||^{-2/3}_{L^1(\mathbb{R})}$, we obtain:
\begin{equation*}
\|\psi\|^2_{L^2(\mathbb{R})}\le
(\frac{1}{4\pi^2}+2)\|\psi\|^{4/3}_{L^1(\mathbb{R})}\|\psi'\|^{2/3}_{L^2(\mathbb{R})},
\end{equation*}
which implies:
\begin{equation*}
\|\psi\|^2_{L^2(D)}\le
(\frac{1}{4\pi^2}+2) |\Omega|^{-2/3}\|\psi\|^{4/3}_{L^1(D)}\|\psi_x\|^{2/3}_{L^2(D)}.
\end{equation*}
Now, by Cauchy-Schwarz inequality, the Sobolev embedding
$W^{1,2}(D)\hookrightarrow L^4(D)$, and the Poincare-Wirtinger
inequality on $\Omega$, if follows that:
\begin{equation*}
\|k\|^2_{L^2(D)}\le \|k\|^{2/3}_{L^1(D)} \|k\|^{4/3}_{L^4(D)}\le C_\Omega \|k\|^{2/3}_{L^1(D)} \|\na k\|^{4/3}_{L^2(D)}.
\end{equation*}
Therefore, by (\ref{2.6}) and (\ref{2.6a}):
\begin{equation*}
\begin{split}
\|\phi\|^2_{L^2(D)} & \le 2(\|\psi\|^2_{L^2(D)}+ 2\|k\|^2_{L^2(D)})
\\ & \le C_\Omega
\left(\|\phi||^{4/3}_{L^1(D)}\|\nabla\phi\|^{2/3}_{L^2(D)} + \|\phi\|^{2/3}_{L^1(D)} \|\na \phi\|^{4/3}_{L^2(D)}\right).
\end{split}
\end{equation*}
We now argue as in \cite{BCR06}. Since
$y:=\|\nabla\phi\|^{2/3}_{L^2}$ satisfies: $ay^2 + by - c \geq 0$ with
appropriate $a,b,c\geq 0$, then: $y\geq \frac{-b + \sqrt{b^2+4ac}}{2a}
= \frac{2c}{b+\sqrt{b^2+4ac}} \geq \frac{c}{\sqrt{b^2+4ac}}$. Hence:
\begin{equation*}
 \|\na\phi\|_{L^2}^{2/3}\ge \|\phi\|^2_{L^2}\left(\|\phi||_{L^1}^{8/3}
   + \|\phi\|_{L^1}^{2/3}\|\phi\|^2_{L^2}\right)^{-1/2},
\end{equation*}
which gives:
\begin{equation*}
\|\na\phi\|_{L^2}^{2}\ge
C_\Omega\|\phi\|^6_{L^2}\left(\|\phi\|_{L^1}^{4} + \|\phi\|_{L^1}\|\phi\|^3_{L^2}\right)^{-1},
\end{equation*}
yielding exactly \eqref{nash}.

\smallskip

{\bf 2.} Recall now that:
\begin{equation}\label{2.3}
 \|\f\|_{L^1(D)} \le \|\f_0\|_{L^1(D)}.
\end{equation}
Indeed, the $L^1$ norm of a solution to \eqref{phi}  is conserved when
the initial data is positive.
In the general case one can write $\f_0=\f_0^+-\f_0^-$ where $\f_0^+$
and $\f_0^-$ are positive,  with disjoint supports. Solving
\eqref{phi} for each, one obtains the inequality (\ref{2.3}).
Further, integrating (\ref{phi}) against $\phi$ and using
incompressibility of $u$ and the boundary condition, it follows that:
\begin{equation}\label{2.4}
\frac{d}{\mbox{d}t}\|\phi\|_{L^2}^2=-2 \|\na\phi\|_{L^2}^2.
\end{equation}
In view of (\ref{nash}), (\ref{2.3}), (\ref{2.4}) we now obtain:
\begin{equation*}
\frac{d\|\phi\|_{L^2}}{\mbox{d}t}
 \le C_\Omega\frac{\|\phi\|^5_{L^2}}{\|\phi_0\|^4_{L^1} + \|\phi_0\|_{L^1}\|\phi\|^3_{L^2}},
\end{equation*}
which, after integrating in time, gives:
\begin{equation}\label{2.18}
t\le C_\Omega\left(\frac{\|\phi_0\|^4_{L^1}}{\|\phi\|^4_{L^2}} + \frac{\|\phi_0\|_{L^1}}{\|\phi\|_{L^2}}\right).
\end{equation}
Call $\alpha = \|\phi\|_{L^2}/ \|\phi_0\|_{L^1}$. From (\ref{2.18}) it
follows that $\frac{\alpha^4}{\alpha^3 + 1}\leq \frac{C_\Omega}{t}$.
The function $\alpha\mapsto\frac{\alpha^4}{\alpha^3 +1}$ is increasing
and it converges to $0$ as $\alpha\to 0$. Let $\beta$ be the unique
solution to $\frac{\beta^4}{\beta^3 + 1} =  \frac{C_\Omega}{t}$, so
that $\alpha\leq \beta$. Now, for $t\to\infty$ clearly
$\frac{C_\Omega}{t}\to 0$, hence also $\beta\to 0$ and 
$\frac{\beta^4}{2} \leq \frac{\beta^4}{\beta^3 + 1} = \frac{C_\Omega}{t}$.
Consequently, $\beta \leq \frac{C_\Omega}{t^{1/4}}$ and we arrive at:
$$ \|\phi\|_{L^2}\leq
\frac{C_\Omega}{t^{1/2}}\|\phi_0\|_{L^1} \qquad \forall t>>1.$$

\smallskip

{\bf 3.} We now argue as in \cite{BCR06}.
Let $\CalP_t$ be the solution operator for \eqref{phi}. Thus far, we have showed that:
\begin{equation*}\label{pt}
\|\CalP_t\|_{L^1\to L^2}\le \frac{C_\Omega}{t^{1/4}} \qquad \forall t>>1.
\end{equation*}
Now if $\CalP_t^*$ is the adjoint operator, then $\CalP_t^*$ is the
solution operator to (\ref{phi}),  with $u$ replaced by $-u$. Therefore the above argument works again:
\begin{equation*}\label{pt*}
\|\CalP^*_t\|_{L^1\to L^2}\le \frac{C_\Omega}{t^{1/4}} \qquad \forall t>>1.
\end{equation*}
Finally, we conclude the lemma:
\begin{equation*}\label{ptt}
\begin{split}
\|\CalP_{2t}\|_{L^1\to L^{\infty}} & \le \|\CalP_t\|_{L^1\to
  L^2}\|\CalP_t\|_{L^2\to L^{\infty}} \\ & =\|\CalP_t\|_{L^1\to
  L^2}\|\CalP^*_t\|_{L^1\to L^2} \leq \frac{C_\Omega}{t^{1/2}} \qquad
\forall t>>1.
\end{split}
\end{equation*}
\end{proof}

We now present a lemma taken from \cite{LM09}:
\begin{lem} \label{supest}
Let $g\in L^2(D)$.
There exists a constant $C_\Omega$, depending only on the crossection
$\Omega$, such that for any solenoidal flow $u\in W^{2,2}\cap
W^{1,2}_0(D)$ satisfying:
\begin{equation*}\label{upf}
-\nu \D u+\na p= g , \qquad \mbox{div } u = 0 \mbox{ in } D,
\end{equation*}
there holds:
\begin{equation}\label{ineq}
\|u\|_{\infty}\le \frac{\sqrt{2}}{\sqrt{\nu\pi}}\|\na u\|^{1/2}_{L^2(D)}
\|g\|^{1/2}_{L^2(D)}+ C_\Omega\|\na u\|_{L^2(D)}.
\end{equation}
\end{lem}

We remark that the proof of (\ref{ineq}) relies on Xie's estimate \cite{Xie92}:
\begin{equation*}\label{xie}
 \|u\|_\infty \le \frac{1}{\sqrt{2\pi}}\|\D u\|_{L^2(D)}^{1/2}\|\na u\|_{L^2(D)}^{1/2}
\end{equation*}
valid for $u\in W^{2,2}(D)\cap W^{1,2}_0(D),$ and on a commutator estimate
\cite{LLP07}:
\begin{equation*}\label{pll}
 \|(\CalP\D-\D\CalP)u\|^2_{L^2(D)}\le \left(\frac12
   +\epsilon\right)\|\D u\|^2_{L^2(D)}  + C_{D,\epsilon}\|\na u\|^2_{L^2(D)}
\end{equation*}
where $\CalP$ is the Helmholz projection onto the space of solenoidal vector
fields. 

The interest in the inequality (\ref{ineq}) lies in the independence of the
constant $\frac{\sqrt{2}}{\sqrt{\pi \nu}}$ at the term involving $g$.
Indeed, this was the key argument allowing to prove \cite{LM09} existence of traveling
waves for the full Navier-Stokes-Boussinesq system in 3d channels,
satisfying appropriate ``thinness'' condition on the crossection
$\Omega$. 
The same argument is needed to obtain the uniform bounds for the stationary
Navier-Stokes-Boussinesq system in Theorem \ref{remi} (ii).

On the other hand, by elliptic estimates and the Sobolev imbedding, it follows directly that:
\begin{equation}\label{easy}
\|u\|_\infty \leq C_\Omega \|u\|_{W^{2,2}(D)} \leq C_\Omega \left(\|\na u\|_{L^2(D)}
+ \frac{1}{\nu}\|g\|_{L^2(D)}\right).
\end{equation}
In fact, already this inequality is sufficient for the 
estimates in case of the Stokes-Boussinesq system.

\bigskip

We will also need the following result, in the line of Lemma 3.6 from \cite{LM09}:
\begin{lem}\label{h}
For each $t$, there exists a function $h\in W^{1,2}_{loc}(D)$, such that:
\begin{equation*}
\|\t(\cdot, t) \ra-\na h\|_{L^2(D)}\le C_\Omega \r \|\na \t(\cdot, t)\|_{L^2(D)},
\end{equation*}
with $C_\Omega$ depending only on  $\Omega$.
In fact:
\begin{equation}\label{const_exact}
C_\Omega = C_{PW}+ \left(\fint_\Omega|\vec g \cdot (0,\tilde
  x)|^2~\mbox{d}\tilde x\right)^{1/2},
\end{equation}
where $C_{PW}$ stands for the Poincare-Wirtinger constant of $\Omega$.
\end{lem}
\begin{proof}
Let $e_1, e_2, e_3$ be the standard basis for $\mathbb{R}^3$.
Suppressing the time variable $t$, we define:
\begin{equation*}
h(x,\tilde x)= \ra \cdot e_1\int_0^x\fint_\Omega\t(x,\tilde x)~\mbox{d}\tilde x ~\mbox{d}x +
\ra \cdot (0,\tilde x)\fint_\Omega\t(x,\tilde x) ~\mbox{d}\tilde x. 
\end{equation*}
Therefore, the following identity concludes the proof of lemma:
\begin{equation*}
\t\ra -\na h= \ra\left(\t(x, \tilde x) - \fint_\Omega\t(x,\tilde x) ~\mbox{d}\tilde x\right)-
\ra \cdot (0,\tilde x) \fint_\Omega\t_x(x,\tilde x)~\mbox{d}\tilde x~ e_1,
\end{equation*}
by the Poincare-Wirtinger inequality: $\|\theta -
\fint \theta \|_{L^2(\Omega)} \leq C_{PW} \|\nabla\theta\|_{L^2(\Omega)}$.
\end{proof}

\section{Proofs of the main result}
The following lemma has been proven in \cite{BCR06} for the case of 2d channels
and vorticity of the flow $u$ vanishing at $\partial D$. Exactly the
same proof, relying on the construction of super and sub-solutions to 
(\ref{heat}) is valid also in the present 3d case. Since the argument
uses the estimate in Lemma \ref{estymata}, we partially reproduce it
below for the sake of completeness.

\begin{lem} \label{1stlem}
There exists a constant $C_{\Omega,0}$ depending on $\Omega$, $f$,
and on the initial data $\t_0$, such that:
\begin{equation}\label{1st}
\bn(t) \le \frac 12 \bb(t) + \bu(t) + C_{\Omega, 0}
\left(\frac1t+\frac{1}{\sqrt{t}}\right) \qquad \forall t>>1.
\end{equation}
Moreover, there exists $x_0>0$ and $q\in L^1(\mathbb{R})$, such that
for all sufficiently large $t>>1$:
\begin{equation}\label{uplowbound}
\begin{split}
& \Phi\left(x-c_0t+x_0+\bar U(t)t + C_{\Omega,0}\sqrt{t}\right) -
Q(x,\tilde x, t)
\leq \theta(x, \tilde x, t) \\ 
& \qquad\qquad \qquad \leq 
\Phi\left(x-c_0t-x_0 - \bar U(t)t - C_{\Omega,0}\sqrt{t}\right) +
Q(x,\tilde x, t),
\end{split}
\end{equation}
where $Q$ is the solution to:
\begin{equation}\label{form}
\begin{split}
& Q_t + u\cdot\nabla Q - \Delta Q = 0 \quad \mbox{ in } D,\\
& \frac{\partial Q}{\partial \vec n} = 0 \mbox{ on } \partial D, \qquad 
Q(x,\tilde x, 0) = q(x).
\end{split}
\end{equation}
\end{lem}
\begin{proof}
We only prove (\ref{uplowbound}), since it implies (\ref{1st}) as in
\cite{BCR06}, Lemma 4.2. Define:
$$\psi_l(x, \tilde x, t) = \Phi\left(x-c_0t+x_0+\bar U(t)t +
  C\sqrt{t}\right) - Q(x,\tilde x, t),$$
where $x_0, C>0$ are to be determined later, and $Q$ is as in (\ref{form}) 
with $q$ appropriately chosen. 

To prove that
$\psi$ is a subsolution, we first need to show the non-positivity of
the following expression:
\begin{equation}\label{2.27}
\begin{split}
& (\psi_l)_t + u\cdot\na \psi_l - \D \psi_l - f(\psi_l)\\
& = \left( \|u(\cdot, t)\|_{L^\infty(D)} + \frac{C}{2\sqrt{t}} +
  u_1(x,\tilde x, t)\right)
\Phi'\left(x-c_0t+x_0+\bar U(t)t + C\sqrt{t}\right) \\ 
& \qquad\qquad\qquad\qquad \qquad\qquad\qquad\qquad 
\qquad\qquad\qquad
+ f(\Phi) - f(\Phi - Q) \\ 
& \leq \frac{C}{2\sqrt{t}}\Phi' +  f(\Phi) - f(\Phi - Q).
\end{split}
\end{equation}
Take  $q\in L^1(\mathbb{R})$ such that $0\leq q(x)\leq \alpha = \min\left(\vartheta_0/2,
  (1-\vartheta_0)/4\right)$. By the maximum principle we have:
\begin{equation}\label{2.28}
0\leq Q(x,\tilde x, t) \leq \alpha.
\end{equation}
Let now $x_0>0$ be such that:
\begin{equation}\label{2.29}
\theta_0(x,\tilde x, 0) \geq \Phi(x+x_0) - q(x).
\end{equation}
Finally, let $C$ be large enough so that when $\Phi\in \left(\vartheta_0,
  1 - (1-\vartheta_0)/4\right)$ then:
\begin{equation*}
\begin{split}
\frac{C}{2\sqrt{t}}\Phi' & + f(\Phi)  - f(\Phi - Q)  \leq
\frac{C}{2\sqrt{t}}\Phi' + \|f'\|_{L^\infty} \|Q(\cdot, t)\|_{L^\infty(D)}
\\ & \leq - \frac{C}{2\sqrt{t}} ~ \min_{\Phi(s) \in (\vartheta_0,
  (3+\vartheta_0)/4)} |\Phi'(s)| ~ + \frac{C_\Omega |\Omega|}{\sqrt{t}} 
\|f'\|_{L^\infty}\|q\|_{L^1(\mathbb{R})} \leq 0,
\end{split}
\end{equation*}
where we used Lemma \ref{estymata} to estimate $\|Q\|_\infty$.
On the other hand, for $\Phi\in(0,\vartheta_0)\cup (1-(1-\vartheta_0)/4,1)$  the
nonpositivity of the right hand side in (\ref{2.27}) follows directly,
via (\ref{2.28}). Concluding, (\ref{2.27}) and (\ref{2.29}) imply the lower bound in (\ref{uplowbound}).

\medskip

The upper bound follows by assuring that:
$$\psi_r(x, \tilde x, t) = \Phi\left(x-c_0t - x_0 - \bar U(t)t -
  C\sqrt{t}\right) + Q(x,\tilde x, t)$$
is a supersolution. Similarly as above, this follows by choosing
$x_0$ such that:
$$\theta_0(x,\tilde x, 0) \leq \Phi(x-x_0) + q(x)$$
in addition to (\ref{2.29}), and having $C$ large enough so that: 
\begin{equation*}
\begin{split}
(\psi_r)_t + u\cdot&\na \psi_r - \D \psi_r - f(\psi_r)
\geq - \frac{C}{2\sqrt{t}}\Phi' -  \|f'\|_{L^\infty} \|Q(\cdot,
t)\|_{L^\infty(D)} \\ & \geq \frac{C}{2\sqrt{t}} ~ \min_{\Phi(s) \in (\vartheta_0/2,
  (1+\vartheta_0)/2)} |\Phi'(s)| ~ - \frac{C_\Omega |\Omega|}{\sqrt{t}} 
\|f'\|_{L^\infty}\|q\|_{L^1(\mathbb{R})} \geq 0,
\end{split}
\end{equation*}
when $\Phi \in (\vartheta_0/2, 1-(1-\vartheta_0)/2)$.
\end{proof}

Clearly, Theorem \ref{secondtheo} follows from (\ref{uplowbound}) and
Lemma \ref{estymata}. We now have:

\begin{lem} \label{2ndlem}
There exists a constant $C_{\Omega,0}$ depending on $\Omega$, $f$, and
on the initial data $\t_0$, such that for all sufficiently large $t>>1$:
\begin{equation}\label{2nd}
\bb(t) \le c_0 + \bu(t) +
C_{\Omega,0}\left(\frac1t+\frac{1}{\sqrt{t}}\right),
\end{equation}
and
\begin{equation}\label{3rd}
\bb(t) \ge c_0 - \bu(t) -
C_{\Omega,0}\left(\frac1t+\frac{1}{\sqrt{t}}\right).
\end{equation}
\end{lem}
\begin{proof}
Denoting $\xi(t) = \bar U(t) t + C_{\Omega,0}\sqrt{t}$and using the
bound  (\ref{uplowbound}), we obtain:
\begin{equation}\label{3.12}
\begin{split}
\bb(t)&=\frac{1}{|\Omega|t}\int_0^t\int_D f(\t) 
=\frac{1}{|\Omega|t}\int_0^t\int_D \t_t ~\mbox{d}x~\mbox{d}\tilde x
\\ &  =\frac{1}{|\Omega|t}\int_D \t(x,\tilde x,t)-\t(x,\tilde x,0) ~\mbox{d}x~\mbox{d}\tilde x\\
&\leq \frac{1}{|\Omega|t}\int_D \Phi(x-c_0t-x_0-\xi(t))+Q(x,\tilde x,
t)-\Phi(x+x_0)+Q(x,\tilde x, 0)~\mbox{d}x~\mbox{d}\tilde x.
\end{split}
\end{equation}
By (\ref{2.3}) and the construction of the corrector $Q$, it follows that:
\begin{equation}\label{ab}
\frac{1}{|\Omega|t} \int_D |Q(x,\tilde x, t) + Q(x, \tilde x, 0)|
~\mbox{d}x~\mbox{d}\tilde x \leq \frac{2}{t} \|q\|_{L^1(\mathbb{R})}
\leq \frac{C_{\Omega, 0}}{t}. 
\end{equation}
Further:
\begin{equation}\label{abc}
\begin{split}
\frac{1}{|\Omega|t} \int_D & \Phi(x-c_0t-x_0-\xi(t)) -\Phi(x+x_0) \\ &
=
\frac{1}{t} \Bigg[\int_{-\infty}^{c_0t+x_0+\xi(t)} (\Phi(x-c_0t-x_0-\xi(t))-1)\\
& \qquad + \int_{-\infty}^{-x_0} (1-\Phi(x+x_0)) + 
\int_{-x_0}^{c_0t+x_0+\xi(t)} (1-\Phi(x+x_0))\\
& \qquad + \int_{c_0t+x_0+\xi(t)}^{\infty} \Phi(x-c_0t-x_0-\xi(t)) \\
& \qquad - \int_{c_0t+x_0+\xi(t)}^{-x_0} \Phi(x + x_0) 
- \int_{-x_0}^\infty\Phi(x+x_0)\Bigg]\\
& = \frac{1}{t} \int_{-x_0}^{c_0t+x_0+\xi(t)} 1 ~\mbox{d}s = \frac{|c_0t + 2x_0 +
  \xi(t)|}{t} \\ & \leq c_0 + \bar U(t) + C_{\Omega, 0} \left(\frac{1}{t} + \frac{1}{\sqrt{t}}\right),
\end{split}
\end{equation}
which together with (\ref{ab}) proves (\ref{2nd}).
To prove (\ref{3rd}), similarly as in (\ref{3.12}) we note that:
\begin{equation*}
\bar B(t) \geq \frac{1}{|\Omega|t}\int_D \Phi(x-c_0t + x_0 +\xi(t)) - Q(x,\tilde x,
t)-\Phi(x-x_0) - Q(x,\tilde x, 0)~\mbox{d}x~\mbox{d}\tilde x,
\end{equation*}
and as in (\ref{abc}) we obtain:
\begin{equation*}
\begin{split}
\frac{1}{|\Omega|t} \int_D &\Phi(x-c_0t+x_0+\xi(t)) -\Phi(x-x_0) =
\frac{|c_0t - 2x_0 - \xi(t)|}{t}\\ & 
\geq c_0 - \bar U(t) - C_{\Omega, 0} \left(\frac{1}{t} + \frac{1}{\sqrt{t}}\right).
\end{split}
\end{equation*}
Together with (\ref{ab}) the above implies (\ref{3rd}).
\end{proof}

\bigskip

\noindent {\bf Proof of Theorem \ref{main}.}\\
Multiplying the fluid equation \eqref{fluid} by $u$ and integrating over $D$, one obtains:
\begin{equation}\label{help0}
\frac{\mbox{d}}{\mbox{d}t} \|u\|_{L^2}^2 + \nu\|\na u\|_{L^2}^2 \leq
C_\Omega \frac{\rho^2}{\nu}\|\nabla \theta\|^2_{L^2}.
\end{equation}
Integrating (\ref{fluid}) against $u_t$ gives, in turn:
\begin{equation*}
\|u_t\|_{L^2}^2  + \nu \frac{\mbox{d}}{\mbox{d}t} \|\nabla u\|_{L^2}^2 \leq
C_\Omega \rho^2\|\nabla \theta\|^2_{L^2},
\end{equation*}
where in both inequalities above we used Lemma \ref{h}
to ``replace'' the term $\theta\vec\rho$ by $\rho\nabla \theta$,
Taking averages in time, we get:
\begin{equation}\label{help}
\begin{split}
 & \frac{1}{t}\int_0^t \|\na u\|_{L^2}^2 ~\mbox{d}t\leq
 C_\Omega\frac{\r^2}{\nu^2}N(t) + \frac{1}{\nu t} \|\na
 u_0\|^2_{L^2},\\
&  \frac{1}{t}\int_0^t \|u_t\|_{L^2}^2 ~\mbox{d}t\leq
 C_\Omega \r^2 N(t) + \frac{\nu}{t} \|\na u_0\|^2_{L^2}.
\end{split}
\end{equation}
By Lemma \ref{supest} or (\ref{easy}), and Lemma \ref{h} it now follows
that:
\begin{equation*}
\begin{split}
\|u\|_\infty &\le C_\Omega \left(\|\nabla u\|_{L^2} + \frac{1}{\nu}\|u_t\|_{L^2}
  + \frac{\rho}{\nu} \|\na \theta\|_{L^2}\right).
\end{split}
\end{equation*}
Using (\ref{help}) we obtain:
\begin{equation}\label{help2}
 \begin{split}
\bar U(t)&\le C_\Omega\left( 
\left(\frac{1}{t} \int_0^t
\|\nabla u\|_{L^2}^2\right)^{1/2} +
\frac{1}{\nu}\left(\frac{1}{t}\int_0^t\|u_t\|_{L^2}^2\right)^{1/2}
+ \frac{\rho}{\nu}\sqrt{N(t)}\right)\\ & 
\leq C_\Omega\left( \frac{\rho}{\nu} \sqrt{N(t)} + 
\frac{1}{\sqrt{\nu t}} \|\nabla u_0\|_{L^2}  +
\frac{1}{\sqrt{t}} \|\nabla u_0\|_{L^2}\right).
\end{split}
\end{equation}
By (\ref{1st}) and (\ref{2nd}) we hence get, for large $t>>1$:
\begin{equation}\label{help3}
 \begin{split}
\bar N(t) & \le \frac{1}{2} c_0 + 
\frac{3}{2}\bar U(t) + C_{\Omega, 0}\left(\frac{1}{t} + \frac{1}{\sqrt{t}}\right)\\ & 
\leq \frac{1}{2} c_0 + C_\Omega \frac{\rho}{\nu} \sqrt{N(t)} + 
C_{all}\left( \frac{1}{t} + \frac{1}{\sqrt{t}} \right).
\end{split}
\end{equation}
where $C_{all}$ depends on $\Omega, f, \nu, \theta_0$ and $u_0$.
Consequently:
\begin{equation*}
 \begin{split}
\bar N(t) & \le \left( C_\Omega \frac{\rho}{\nu} + \sqrt{\frac{1}{2} c_0 + 
 C_\Omega \frac{\rho^2}{\nu^2}} + C_{all} \sqrt{\frac{1}{t} +
\frac{1}{\sqrt{t}} }~\right)^2\\
& \leq \left(C_\Omega \frac{\rho}{\nu} +
\sqrt{\frac{1}{2} c_0 +  C_\Omega
    \frac{\rho}{\nu}}\right)^2 
+ C_{all} \left(\frac{1}{t} + \frac{1}{\sqrt{t}} \right),
 \end{split}
\end{equation*}
and, returning to (\ref{help2}):
$$\bar U(t) \leq C_\Omega\left( \frac{\rho^2}{\nu^2} +
  \frac{\rho}{\nu} \sqrt{\frac{1}{2} c_0 +  C_\Omega
    \frac{\rho}{\nu}}\right) + C_{all} \left(\frac{1}{t} + \frac{1}{\sqrt{t}} \right). $$
In view of Lemma \ref{2ndlem}, Theorem \ref{main} is hence proven.
\endproof

\bigskip

\noindent {\bf Proof of Theorem \ref{remi}.}\\
We only prove the assertion (ii), because the 2d case in (i) follows with exactly
the same calculations as in Theorem \ref{main} and Theorem \ref{secondtheo} .

For the stationary Navier-Stokes-Boussinesq system, 
using the Poincar\'e inequality and
(\ref{const_exact}), we obtain the following counterpart of (\ref{help0}):
\begin{equation}\label{help11}
\|\nabla u\|_{L^2(D)} \leq \frac{\rho}{\nu} C_P\left(C_{PW} +
  \left(\fint|\vec g\cdot (0,\tilde x)|^2~\mbox{d}\tilde x\right)^{1/2}
\right) \|\nabla \theta\|_{L^2(D)},
\end{equation}
where $C_P$ and $C_{PW}$ are, respectively, the Poincar\'e and the
Poincar\'e-Wirtinger constants of $\Omega$.
By Lemma \ref{estymata} and argueing as (\ref{help2}), we arrive at:
\begin{equation*}
\bar U(t) \leq \frac{\rho^2}{\nu^3} \frac{|\Omega|}{2\pi} C_P^2 \left(C_{PW} +
  \left(\fint|\vec g\cdot (0,\tilde x)|^2~\mbox{d}\tilde x\right)^{1/2}
\right)^2 \bar N(t) + C_\Omega \frac{\rho}{\nu} \sqrt{\bar N(t)},
\end{equation*}
while the counterpart of (\ref{help3}) in the present case is:
\begin{equation*}
\begin{split}
\bar N(t) \leq \frac{1}{2}c_0 & + C_\Omega \frac{\rho}{\nu} \sqrt{\bar N(t)}
+ C_{\Omega, 0} \left(\frac{1}{t} + \frac{1}{\sqrt{t}}\right) \\
& + \frac{3|\Omega|}{4\pi}\frac{\rho^2}{\nu^3} C_P^2 \left(C_{PW} +
  \left(\fint|\vec g\cdot (0,\tilde x)|^2~\mbox{d}\tilde x\right)^{1/2}
\right)^2\bar N(t) 
\qquad
\forall t>>1.
\end{split}
\end{equation*}
It is therefore clear that when the constant in front of $\bar N(t)$
in the last term of the right hand side above is smaller than $1$, the
results of Theorem \ref{main} and Theorem \ref{secondtheo} 
follow as in the case of the Stokes-Boussinesq system.
\endproof

\end{document}